\documentclass{mathscan}

\usepackage[utf8]{inputenc}
\usepackage{amssymb} %amsmath,amsthm
\usepackage{xypic}
\usepackage{comment}
\usepackage{hyperref}
\usepackage{enumerate}
\usepackage{color}
\definecolor{darkgreen}{RGB}{47,139,79}
\definecolor{darkblue}{RGB}{36,24,130}
\usepackage{hyperref}
\hypersetup{
    colorlinks=true, %set true if you want  colored links
    linktoc=all,     %set to all if you want both sections and subsections linked
    linktocpage,
    citecolor=darkgreen,
    filecolor=black,
    linkcolor=darkblue,
    urlcolor=black
}%%%%%%%%%%%%%%%%%%%%%%%%%%%%

\newcommand{\im} {\operatorname{Im}}

\newcommand{\Aut} {\operatorname{Aut}}

\newcommand{\inv} {\ensuremath ^{-1}}

\newcommand{\F} {\ensuremath{\mathbb{F}}}

\newcommand{\set}[2] {\left\lbrace {#1} \ \middle\arrowvert\ {#2} \right\rbrace}

\newcommand{\La}{\Lambda}

\DeclareMathOperator\id{id}

\DeclareMathOperator\G{\mathrm{GL}}
\DeclareMathOperator\Sp{\mathrm{Sp}}
\DeclareMathOperator\OO{\mathrm{O}}
\DeclareMathOperator\UU{\mathrm{U}}

\DeclareMathOperator\Gi{\mathbf{G}}
\DeclareMathOperator\Ai{\mathbf{A}}

\DeclareMathOperator\PPi{\mathbf{P}}

\DeclareMathOperator\R{\mathcal{R}} %% FOR Rad 
\DeclareMathOperator\K{\mathcal{K}} %% for kernel = V^\perp

\newcommand{\Forms}{\operatorname{Form}}

\newtheorem{thm}{Theorem}[section]
\newtheorem{prop}[thm]{Proposition}
\newtheorem{lemma}[thm]{Lemma}
\newtheorem{lem}[thm]{Lemma}
\newtheorem{cor}[thm]{Corollary}

\newtheorem{Th}{Theorem}

\theoremstyle{definition}
\newtheorem{rem}[thm]{Remark}
\newtheorem{defn}[thm]{Definition}
\newtheorem{Def}[thm]{Definition}
\newtheorem{ex}[thm]{Example}

\numberwithin{equation}{section}

\newcommand{\s}{\sigma}
\newcommand{\om}{\omega}

\newcommand{\eps}{\varepsilon}

\newcommand{\op}{\oplus}
\newcommand{\x}{\times}

\newcommand{\rar}{\longrightarrow}
\newcommand{\sta}{\stackrel}

\newcommand{\note}[1]%{}
{{\bf [N: #1]}}

\begin{document}

\title{Forms over fields and Witt's lemma}
\author{David Sprehn}
\address{Department of Mathematics\\ 
University of Copenhagen\\ 
Universitetsparken 5\\ 
2100 Copenhagen\\ 
Denmark}
\email{david.sprehn@gmail.com}

\author{Nathalie Wahl}
\address{Department of Mathematics\\ 
University of Copenhagen\\ 
Universitetsparken 5\\ 
2100 Copenhagen\\ 
Denmark}
\email{wahl@math.ku.dk}

\begin{abstract}
We give an overview of the general framework of  forms of Bak, Tits and Wall, when restricting to vector spaces over fields, and describe its relationship to the classical notions of Hermitian, alternating and quadratic forms. We then prove 
a version of Witt's lemma in this context, showing in particular that the action of the group of isometries of a space equipped with a form is transitive on isometric subspaces. 
\end{abstract}

\maketitle

\section{Introduction}

Let $\F$ be a (commutative) field, and $\s:\F\to \F$  an involution, that is a field isomorphism squaring to the identity. 
To simplify notations, we will write $\bar x$ for $\s(x)$ in analogy with complex conjugation. 
Let $E$ be an $\F$--vector space. Recall that a map $f:E\times E\to\F$ is called {\em sesquilinear} if it is biadditive and 
satisfies that  $f(av,bw)=\bar af(v,w)b$ for all $a,b\in \F$ and
$v,w\in E$. 
If $\s$ is the identity, such a map is just a bilinear form.

The ``classical groups'' over $\F$ are  the automorphism groups ({\em isometry groups}) of different types of forms on $\F$--vector spaces:

\smallskip

\noindent
(1) 
A symplectic group is the group of isometries of a vector
space $E$ equipped with an {\em alternating form}, that is a bilinear map $\omega:E\times E\to\F$ such that 
\[ \omega(v,v)=0.\]

\smallskip

\noindent
(2) 
A  unitary group is the group of isometries of a vector space $E$ equipped with a 
{\em Hermitian form}, that  is a map $\omega:E\times E\to\F$ that is sesquilinear  with respect to a (non-trivial) involution $\s:\F\to\F$, and such that 
\[ \om(v,w)=\overline{\om(w,v)}. \]

\smallskip

\noindent
(3) 
An orthogonal group  is the group of isometries of a vector space $E$ equipped with a {\em quadratic form}, that is a set map $Q:E\to \F$
such that
\[ Q(cv)=c^2Q(v)\textrm{ for all }c\in\F,\ v\in E \]
and such that the associated map
$B_Q(v,w)=Q(v+w)-Q(v)-Q(w)$
is bilinear.

These examples of ``forms'' all fit into a common framework,
 developed in particular by Bak, Tits, Wall (see eg.~\cite{Bak69,Tits,Wallaxiomatic,WallII}) and later by Magurn-Vaserstein-van der Kallen \cite[Sec~4]{MVV}. This framework 
 allows to treat all three cases at once,  including when the field has characteristic  2, and symmetric and  anti-symmetric forms are indistinguishable. 
The framework makes sense in the context of rings with
anti-involutions, and has turned out very useful  for example for 
proving homological stability of symplectic, unitary and orthogonal groups in one go, see eg.~\cite{Frie,Mir,RWW} in the context of rings, and \cite{Classical} in the special case of fields, or in the study of certain classes of subgroups of general linear groups, see eg.~\cite{Petrov}. 

In the present paper, we give an overview of what this general framework becomes under the assumption that we work only with fields instead of general rings. In particular, we 
 assemble and complete results from the literature to show that in this case, alternating, Hermitian and quadratic forms are essentially the only existing flavors of forms arising in this context. 
In the second part of the paper, we
prove a general version  of the classical Witt's Lemma \cite{Witt}, showing that the isometry group acts transitively on isometric subspaces. We allow the case of degenerate forms that is often not treated in the literature, and provide a relative version. This last result is  used by our companion paper \cite{Classical}, which proves a homological stability result for such isometry groups of vector spaces equipped with forms.

\medskip

We describe our main results in more detail now.

Fix a pair $(\F,\s)$ with $\F$ a field and $\s$ an involution (possibly the identity). 
The generalized  definition of a {\em form} over an $\F$--vector space $E$, which we describe now,  depends on the further choice of 
an element $\eps\in \F$ such that $\eps\bar\eps=1$,  and a certain additive
subgroup $0\le \Lambda\le \F$. In the present situation, with $\F$ a field, the group $\Lambda$ is almost always determined by the triple $(\F,\s,\eps)$ (see Proposition~\ref{prop:Lambda} in the Appendix).

Given such a tuple $(\F,\s,\eps,\La)$ and a sesquilinear form $q:E\x E\to \F$, we can construct two new maps:  
define $\omega_q:E\x E\to \F$ and $Q_q:E\to \F/\Lambda$ by setting
 \[ \omega_q(v,w)=q(v,w)+\eps\overline{q(w,v)} \ \ \ \textrm{and}\ \ \  Q_q(v)=q(v,v)+\Lambda\in\F/\Lambda. \]
The map $\omega_q$ is an ``$\eps$-skew symmetric sesquilinear form'' by construction, and $Q_q$ is to be thought of as  a ``quadratic refinement'' of $\omega_q$.  
 Theorem~\ref{thm:qdetermined} shows that $\omega_q$ and $Q_q$ are always related by the equation 
$$Q_q(v+w)-Q_q(v)-Q_q(w)=\omega_q(v,w) \in \F/\Lambda.$$ 
In particular, if $\s=\id$, $\eps=1$ and $\La=0$, then $Q_q$ is precisely a quadratic form as described above, with $\omega_q$ its associated bilinear form.
But if $\s\neq \id$ with $\eps=1$, then $\omega_q$ is a Hermitian form. To get alternating forms, we will take $\s=\id$ and $\eps=-1$.

The association $q \mapsto (\omega_q,Q_q)$, from sesquilinear maps to such pairs of maps, is not injective.  
Let $X(E,\sigma,\eps,\Lambda)\le \operatorname{Sesq}_\sigma(E)$ denote its kernel, where 
$\operatorname{Sesq}_\sigma(E)$ denotes the vector space of sesquilinear forms on $E$. (See Definition~\ref{def:form} for a direct description of $X$.)
 The group of all {\em $(\s,\eps,\Lambda)$--quadratic forms on $E$} is  defined to be the quotient 
\[\Forms(E,\s,\eps,\Lambda):= \operatorname{Sesq}_{\sigma}(E)/X(E,\sigma,\eps,\Lambda).  \]
By definition, the form $q$ is thus always determined by the pair $(\omega_q,Q_q)$ 
and in fact, Proposition~\ref{prop:inject} shows that only one of $\omega_q$ or $Q_q$ is always enough to determine $q$, but which one of the two depends on the chosen $\Lambda$. 
Hence, instead of working sometimes with forms like $\omega_q$, sometimes with forms like $Q_q$, we can always just work with equivalence classes of sesquilinear forms $q$. 
This leaves us with the question of what ``forms like $\omega_q$ or $Q_q$'' arise in this framework? This is answered by 
our first main result, which says that, for almost all choices of parameters $(\s,\eps,\Lambda)$, the general framework of forms simply specializes to the classical Hermitian, alternating and quadratic forms: 

\begin{Th}\label{introthm:qspecialcases}
(1)~When $\s\neq \id$, 
there is a natural isomorphism 
 $$\Forms(E,\sigma,\eps,\Lambda) \cong \textrm{Herm}_{\s}(E).$$ 

\noindent (2)~When $\s=\id$, $\eps=-1$ and $\Lambda=\F$, 
there is a natural isomorphism 
 $$\Forms(E,\sigma,\eps,\Lambda) \cong \textrm{Alt}(E).$$ 

\noindent (3)~When  $\s=\id$, $\eps=1$ and $\Lambda=0$, there is a natural isomorphism 
$$\Forms(E,\s,\eps,\Lambda) \cong  \textrm{Quad}(E).$$
\end{Th}
Here $\textrm{Herm}_{\s}(E)$, $\textrm{Alt}(E)$ and $\textrm{Quad}(E)$ are the groups of Hermitian, alternating and quadratic forms as described above. (See also Definition~\ref{Def:classical}.)
When $\F$ is a finite field, there is only one isomorphism class of non-degenerate such form in each dimension, except in the orthogonal case (3) (see eg.~\cite[Chap II]{FP}). Forms over infinite fields is on the other hand a vast subject, see eg.~the book \cite{Lam} which is concerned with quadratic forms over fields of characteristic not 2.

\begin{rem}\label{rem:almostall} (i) When $\s\neq \id$, the result shows that the group of forms depends neither on $\Lambda$, 
which is a consequence of the fact that there is only one possible $\Lambda$ in this case by Proposition~\ref{prop:Lambda}, nor on $\eps$.

\smallskip

\noindent
(ii) When $\s=\id$, there are two main possibilities: the alternating case (2), and the quadratic case (3). 
If the field is not of characteristic 2,  cases (2) and (3) are distinguished by whether $\eps=1$ or $\eps=-1$ and $\Lambda$ is determined by this data. 
If $\F$ has characteristic 2, then $1=-1$ and cases (2) and (3) in the theorem are distinguished by the choice of $\Lambda$. This is in fact the one situation 
where there is a freedom of choice for $\Lambda$ 
and cases (2) and (3) correspond to the two possible extremes, the smallest and largest possible $\Lambda$. 
Note also that, given $\Lambda$, Proposition~\ref{prop:Lambda} shows that $\eps$ is in fact always redundant information when $\s=\id$. We chose to leave the $\eps$ in the statement as it is informative. 

\smallskip

\noindent
(iii) The only cases not covered by the theorem are thus in characteristic 2 with $\s=\id$ and $0<\Lambda<\F$. 
Proposition~\ref{prop:perfectfield} shows that under the additional assumption that $\F$ is a perfect field, also these cases are actually covered by (3). More generally, combining Proposition~\ref{prop:inject}\,(2) and Lemma~\ref{lem:surject}\,(2), we get that a ``mod $\Lambda$'' version of (3) holds under the weaker assumption that $\s=\id$ and $\Lambda<\F$, giving the remaining cases. 
\end{rem}

The above results are all proved in Section~\ref{sec:forms}.

\medskip

The second part of the paper is concerned with isometries: 
 Given a vector space $E$ equipped with a form $q$, called here a {\em formed space} $(E,q)$, an {\em isometry} of $E$ is a linear map respecting the form.  
For appropriate choices of $q$, the classical groups $\Sp_{2n}(\F)$, $\OO_n(\F)$, $\OO_{n,n}(\F)$, $\UU_n(\F)$ and $\UU_{n,n}(\F)$ are all such isometry groups (see Example~\ref{ex:groups}). 
Restricting the form $q$ to subspaces of $E$,  we can talk about isometries between subspaces. 

The classical groups just mentioned are all automorphism groups of {\em non-degenerate}  formed spaces, that is with trivial kernel, 
where the kernel $\K(E)$ of $(E,q)$ is defined to be the kernel of the map 
\[ E\to {^\sigma\! E^*},\ v\mapsto\omega_q(-,v), \]
where ${^\sigma\! E^*}$ denotes the vector space of $\s$-skew-linear maps $E\to \F$. 
Degenerate formed spaces do however also naturally occur (see
eg.~Example~\ref{ex:Cox} for  examples arising from Coxeter groups,
or, in the context of rings, \cite[Sec 5]{GRWI} for examples coming
from intersections of immersed spheres on manifolds).  
In the present paper, we do  allow forms with non-trivial kernels.  

Our main result in the second part of the paper is the following:

\begin{Th}[Witt's Lemma (Theorem~\ref{lem:Witt})]
Let $(E,q)$ be a formed space.
Suppose that $f:U\to W$ is a bijective isometry between two subspaces of $E$ such that $f\big(U\cap\K(E)\big)= W\cap\K(E)$. 
Then $f$ can be extended to a bijective isometry of $E$. 
\end{Th}

Theorem~\ref{thm:relativeWitt} in the paper gives in addition a relative version of Witt's lemma, where  a subspace may be assumed to be fixed by the constructed isometry. 

Witt originally considered in \cite{Witt} the case of non-degenerate symmetric forms in characteristic not 2.  His result was generalized by many authors (see eg.~\cite[Prop 2]{Tits} and \cite[p21,22,36]{Dieudonne}, \cite[Cor 8.3]{MVV}, \cite[Thm 3.9]{Artin}, \cite[Thm 1]{Petrov} and the paper \cite{Huang} for many references to the existing literature). 
We were however not able to find the result in the above generality, as needed by our companion paper \cite{Classical}.

The proofs of Witt's lemma and its relative version are given in Section~\ref{sec:Witt}.

\section{Forms}\label{sec:forms}

Let $(\F,\s)$, as above, be a field with a chosen involution. Examples
we have in mind are $\F$ any field with $\s=\id$, or $\F=\mathbb{C}$ with $\s$ the conjugation, or $\F=\F_{p^{2r}}$ a finite field with $\s=(-)^r$.  As before, we write  $\bar c=\sigma(c)$.

Pick $\eps\in\F$ satisfying $\eps\bar\eps=1$. Taking $\eps=\pm 1$ is always a possibility, but when the involution is non-trivial, there may be many more choices. 
For example $\eps=i$ also works when $(\F,\s)$ is the complex numbers with conjugation. 

Set 
\begin{align*} \Lambda_{min} &= \set{a-\eps\bar a}{a\in\F} \\
 \Lambda_{max} &= \set{a\in\F}{a+\eps\bar a=0}. \end{align*}
These are additive subgroups of $\F$. Let $\Lambda$ be an abelian group such that
\[ \Lambda_{min}\leq\Lambda\leq\Lambda_{max}. \]
Proposition~\ref{prop:Lambda} in the appendix shows that for most triples $(\F,\s,\eps)$, we actually have $\Lambda_{min}=\Lambda_{max}$ so 
that there is a unique $\Lambda$ determined by the triple. Typical values of $\La$ are $0$, the fixed points 
$\F^{\s}$ of the involution, or the whole field $\F$.

\begin{Def}\label{def:form}
A $(\sigma,\eps,\Lambda)$-\textit{quadratic form} on $E$ (or just a {\em form} for short) is an element of the quotient
\[\Forms(E,\s,\eps,\Lambda):= \operatorname{Sesq}_{\sigma}(E)/X(E,\sigma,\eps,\Lambda),  \]
of the vector space  $\operatorname{Sesq}_\sigma(E)$ of sesquilinear forms on $E$
by  its the additive subgroup  $X(E,\sigma,\eps,\Lambda)$ of all those $f$ satisfying
\[ f(v,v)\in\Lambda  \ \ \ \textrm{and}\ \ \   f(w,v)=-\eps\overline{f(v,w)} \]
for all $v,w\in E$. 
A \textit{formed space} $(E,q)$ is a finite-dimensional vector space $E$ equipped with a form $q$.
\end{Def}
We will see in Theorem~\ref{thm:qdetermined} below that this agrees with the definition given in the introduction, where $X$ was defined as the kernel of a map. 

\begin{rem}
A typical example of an element in $X(E,\sigma,\eps,\Lambda)$ has the following form: let $q\in \operatorname{Sesq}_{\sigma}(E)$ be a sesquilinear form, and let 
$$f(v,w)=q(v,w)-\eps \overline{q(w,v)}.$$
Then $-\eps\overline{f(v,w)}=-\eps\overline{q(v,w)}+\eps\overline{(\eps \overline{q(w,v)})}=f(w,v)$ given that $\eps\bar \eps=1$, while $f(v,v)\in \La_{min}\le \La$. 
In particular, $X(E,\sigma,\eps,\Lambda)$ is always non-trivial. 
\end{rem}

\begin{ex}\label{ex:forms}
\noindent (1) Let $E=\F^n$. The {\em Eucledian form} on $E$ is the form $q_{\mathcal{E}}$ defined by $$q_{\mathcal{E}}(v,w)=\sum_{i=1}^n\overline{v_i}w_i.$$

\smallskip

\noindent (2) Let $E=\F^{2n}$. The {\em hyperbolic form} on $E$ is the form $q_{\mathcal{H}}$ defined by $$q_{\mathcal{H}}(v,w)=\sum_{i=1}^n\overline{v_{2i-1}}w_{2i}.$$ 
\end{ex}

If $E$ is finite dimensional, say of dimension $n$, the choice of a basis gives a one-to-one correspondence between sesquilinear forms on $E$ and $(n\x n)$--matrices with coefficients in $\F$, with the $(i,j)$--th entry of the matrix being the value of the form on the $i$th and $j$th basis vectors. 
Describing the quotient of sesquilinear forms by $X(E,\sigma,\eps,\Lambda)$ is of course much more tricky. 

\medskip

The set of forms on $E$ is contravariantly functorial in $E$: from a linear map $f:E\to F$ and a form $q$ on $F$, we produce a form on $E$ as
\[ f^*q(v,w)=q(f(v),f(w)) \]
for $v,w\in E$.
This is indeed well-defined as $f^*q$ is again sesquilinear, and $f^*$
takes $X(F,\sigma,\eps,\Lambda)$ to  $X(E,\sigma,\eps,\Lambda)$.

\medskip

The goal of this section is to relate forms in the sense of Definition~\ref{def:form} to the more classical notions of Hermitian, alternating, and quadratic forms.

\begin{Def}\label{Def:classical} Let $E$ be a vector space over $\F$. Define

\smallskip

\noindent
\begin{description}
\item[$\textrm{Herm}_{\s,\eps}(E)$] the vector space of sesquilinear forms $\omega: E\x E\to \F$ satisfying that 
$$\omega(v,w)=\eps\overline{\omega(w,v)}$$ for all $v,w\in E$, considered here as a group. 
We write $\textrm{Herm}_{\s}(E)=\textrm{Herm}_{\s,1}(E)$ for the classical Hermitian forms.

\item[$\textrm{Alt(E)}\le \textrm{Herm}_{id,-1}(E)$] the subspace of alternating forms, i.e.~those such that $$\omega(v,v)=0$$ for all $v\in E$.

\item[$\textrm{Set}(E,\F/\Lambda)$] the group of set maps $Q:E\to \F/\Lambda$.

\item[$\textrm{Quad}(E)\le \textrm{Set}(E,\F)$] the subgroup of classical quadratic forms, i.e. the maps
$Q:E\to\F$ such that $$Q(av)=a^2Q(v)$$ for all $v\in E$, $a\in\F$ and 
such that $B_Q:E\x E\to \F$ defined by  $$B_Q(v,w)=Q(v+w)-Q(v)-Q(w)$$ 
is bilinear. 
\end{description}
\end{Def}

\medskip

 Just like $\Forms(E,\sigma,\eps,\Lambda)$, the groups $\textrm{Herm}_{\s,\eps}(E)$, $\textrm{Alt(E)}$,  $\textrm{Set}_{\s}(E,\F/\Lambda)$ and  $\textrm{Quad}(E,\F)$ 
assemble to give  contravariant functors from the category of vector spaces and linear maps to the category of abelian groups.

\begin{thm}\label{thm:qdetermined}(see also \cite[Thm 1]{Wallaxiomatic})
Let $E$ be a vector space.  For $q\in \Forms(E,\s,\eps,\Lambda)$, setting
 \[ \omega_q(v,w)=q(v,w)+\eps\overline{q(w,v)} \ \ \ \textrm{and}\ \ \  Q_q(v)=q(v,v)+\Lambda\in\F/\Lambda, \]
defines a map 
$$\chi=(\omega_{(-)},Q_{(-)}): \Forms(E,\sigma,\eps,\Lambda) \rar \textrm{Herm}_{\s,\eps}(E)\times \textrm{Set}(E,\F/\Lambda)$$
which is an injective homomorphism,  natural with respect to linear maps $E\to E'$.
In particular, if $(E,q)$ and $(E',q')$ are formed spaces, and $f:E\to E'$ a linear map, then $f$ is an isometry if and only if
\[ \omega_{q'}(f(v),f(w))=\omega_q(v,w)  \ \ \  \ \textrm{and}\ \ \ \ Q_{q'}(f(v))=Q_{q}(v) \]
for all $v,w\in E$.
Moreover, for any $q\in \Forms(E,\s,\eps,\Lambda)$, the maps $\omega_q$ and $Q_q$ are related by the following equation: 
\begin{equation}\label{equ:Qomega}
Q_q(v+w)-Q_q(v)-Q_q(w)=\omega_q(v,w) \in \F/\Lambda.
\end{equation}
\end{thm}

\begin{ex}
(1) When $\s=\id$ and $\eps=1$, the form $Q_{q_{\mathcal{E}}}$ associated to the Eucledian form of Example~\ref{ex:forms} is the standard quadratic form $Q_{q_{\mathcal{E}}}(v)=\sum_{i=1}^nv_i^2$ on $\F^n$, with 
$\omega_{q_{\mathcal{E}}}=2q_{\mathcal{E}}$ its associated symmetric bilinear form. Note that $\omega_{q_{\mathcal{E}}}$ is zero in characteristic 2!

\smallskip 

\noindent (2) When $\s=\id$ and $\eps=-1$, the  form $q_{q_{\mathcal{H}}}$ associated to the hyperbolic form of Example~\ref{ex:forms} is (up to reordering of the basis vectors) the standard symplectic form on $\F^{2n}$, namely 
$\omega_{q_{\mathcal{H}}}(v,w)=\sum_{i=1}^nv_{2i-1}w_{2i}-w_{2i-1}v_{2i}$. 
\end{ex}

\begin{proof}[Proof of Theorem~\ref{thm:qdetermined}]
The map $\chi$ 
is well-defined and injective since
 $\omega_f=0$ and $Q_f(v)\in \Lambda$ if and only if $f\in X(E,\sigma,\eps,\Lambda)$.
Also, for any $q\in\Forms(E,\s,\eps,\Lambda)$, we have that  $\omega(v,w)=\eps \overline{\omega(w,v)}$ for all $v,w\in E$, so the first component of $\chi$ has image in $\textrm{Herm}_{\s,\eps}(E)$ as claimed. 
One checks that the assignment
is a group homomorphism that is natural with respect to vector space homomorphisms.
The characterization of the isometries follows directly from naturality and injectivity of $\chi$. 
Finally, the last equation follows from the fact that $q(w,v)-\eps \overline{q(w,v)}\in \Lambda_{min}$. 
\end{proof}

The group of linear automorphisms $\G(E)$ acts on both the domain and codomain of $\chi$, and naturality implies that $\chi$ commutes with this action.  
The isotropy group
of an element $q\in\Forms(E,\sigma,\eps,\Lambda)$ is the group $\Gi(E,q)$ of bijective isometries of $(E,q)$. It follows from the above result that 
$$\Gi(E,q)=\Aut(E,\omega_q)\cap \Aut(E,Q_q)$$
is exactly the group of maps preserving both $\omega_q$ and $Q_q$. The following result show that in fact one of those two conditions is always redundant and the automorphism group
$\Gi(E,q)$ identifies with either $\Aut(E,\omega_q)$ or $\Aut(E,Q_q)$:

\begin{prop}\label{prop:inject}
(1)~If $\Lambda=\Lambda_{max}$, 
then $Q_q$ is determined by $\omega_q$, in the sense that 
the assignment $$\omega_{(-)}: \Forms(E,\sigma,\eps,\Lambda) \rar \textrm{Herm}_{\s,\eps}(E)$$ 
of Theorem~\ref{thm:qdetermined} is injective.
In particular, $\Gi(E,q)=\Aut(E,\omega_q)$ as subgroups of $\G(E)$ in this case.

\noindent
(2)~If $\Lambda\neq \F$, then
$\omega_q$ is determined by $Q_q$, in the sense that 
the assignment
$$Q_{(-)}:\Forms(E,\sigma,\eps,\Lambda) \rar  \textrm{Set}(E,\F/\Lambda)$$
of Theorem~\ref{thm:qdetermined}
is injective.
In particular, $\Gi(E,q)=\Aut(E,Q_q)$  as subgroups of $\G(E)$ in this case.
\end{prop}

\begin{ex}\label{ex:groups} Recall the Eucledian and hyperbolic forms $q_{\mathcal{E}}$ and $q_{\mathcal{H}}$ of Example~\ref{ex:forms}. 

\smallskip

\noindent
(1) The group $\Gi(\F^{2n},q_{\mathcal{H}})$ for different choices of parameters $(\s,\eps,\La)$ identifies with classical groups as follows: \\
\begin{tabular}{lll}
$(\s,\eps,\La)=(\id,-1,\F)$ & $\Gi(\F^{2n},q_{\mathcal{H}}\!)\!=\!\Aut(\F^n,\omega_{q_\mathcal{H}}\!)$&$=\Sp_{2n}(\F)$\\ % is the symplectic group.\\
 $(\s,\eps,\La)=(\neq \id,1,\F^\s)$ & $\Gi(\F^{2n},q_{\mathcal{H}}\!)\!=\!\Aut(\F^n,Q_{q_\mathcal{H}}\!)\!=\!\Aut(\F^n,\omega_{q_\mathcal{H}}\!)$\!&$=\UU_{n,n}(\F)$\\ %is the unitary group.\\ 
$(\s,\eps,\La)=(\id,1,0)$ & $\Gi(\F^{2n},q_{\mathcal{H}}\!)\!=\!\Aut(\F^n,Q_{q_\mathcal{H}}\!)$&$=\OO_{n,n}(\F)$ %is the orthogonal group.\\
\end{tabular}

\smallskip

\noindent
(2) The group $\Gi(\F^{n},q_{\mathcal{E}})$ for different choices of parameters $(\s,\eps,\La)$ identifies with classical groups as follows: \\
\begin{tabular}{lll}
 $(\s,\eps,\La)=(\neq \id,1,\F^\s)$ & $\Gi(\F^{n},q_{\mathcal{E}})=\Aut(\F^n,Q_{q_\mathcal{E}})=\Aut(\F^n,\omega_{q_\mathcal{E}})$&$=\UU_{n}(\F)$\\ 
$(\s,\eps,\La)=(\id,1,0)$ & $\Gi(\F^{n},q_{\mathcal{E}})=\Aut(\F^n,Q_{q_\mathcal{E}})$&$=\OO_{n}(\F)$ 
\end{tabular}\\
the latter under the additional assumption that $\textrm{char}(\F)\neq 2$ for the form $q_{\mathcal{E}}$ to be non-degenerate. 
The ``symplectic version'' of the Euclidean groups would be the general linear group $\G(E)$ because the form $q_{\mathcal{E}}$ is trivial with the parameters  $(\s,\eps,\La)=(\id,-1,\F)$. 
\end{ex}

\begin{proof}[Proof of Proposition~\ref{prop:inject}]
We first prove (1). Since we know $q\mapsto \chi(q)=(\omega_q,Q_q)$ is additive and injective, it suffices to check that
$\omega_q=0$ implies $Q_q=0$.  Indeed, for $v\in E$,
\[ 0=\omega_q(v,v)=q(v,v)+\eps\overline{q(v,v)}, \]
so
\[ q(v,v)\in\set{a\in\F}{a+\eps\overline{a}=0} = \Lambda_{max}=\Lambda, \]
meaning $Q_q(v)=q(v,v)+\Lambda=0\in \F/\Lambda$, proving (1).

For statement (2), we assume $\Lambda\neq \F$ is a proper additive subgroup. 
By the additivity and injectivity of $\chi$,
it suffices now to check that
$Q_q=0$ implies $\omega_q=0$.  
The condition $Q_q=0$ is equivalent to $q(v,v)\in \Lambda$ for every $v\in E$. It follows that 
\begin{align*}
q(u,v)&= q(u+v,u+v)-q(u,u)-q(v,v)-q(v,u) \\
&= -q(v,u) \mod \Lambda\\
&= -\eps\overline{q(v,u)} \mod \Lambda, 
\end{align*}
where the last equality holds as $q(v,u)-\eps\overline{q(v,u)} \in \Lambda_{min}\le \Lambda$. 
Hence $\omega(u,v)\in \Lambda$ for every $u,v\in E$. Scaling $v$ by a any element in $\F$ defines an ideal of $\Lambda$ generated by $\omega(u,v)$. As $\Lambda$ has no non-trivial ideal, it follows that 
$\omega(u,v)=0$, which proves (2). 
\end{proof}

Finally, we prove Theorem~\ref{introthm:qspecialcases}. The proof  will use two lemmas, which we state and prove now.

\begin{lem}\label{lem:surject}
(1)~The assignment $$\omega_{(-)}: \Forms(E,\sigma,\eps,\Lambda) \rar \textrm{Herm}_{\s,\eps}(E)$$ 
of Theorem~\ref{thm:qdetermined} is surjective 
unless $\s=\id$ and $\F$ has is a field of characteristic $2$, in which case the image is the subspace $\textrm{Alt}(E)$ of alternating bilinear forms. 

\smallskip

\noindent
(2)~If $\sigma=\id$ 
the assignment
$$Q_{(-)}:\Forms(E,\sigma,\eps,\Lambda) \rar  \textrm{Set}(E,\F/\Lambda)$$
of Theorem~\ref{thm:qdetermined} has image the {\em mod $\Lambda$ quadratic forms}, i.e.~the set maps $Q:E\to \F/\Lambda$ such that  $Q(av)=a^2Q(v)$ for all $v\in E$, $a\in\F$ and 
$B_Q(v,w)=Q(v+w)-Q(v)-Q(w)$ is a bilinear function. In particular, if in addition $\Lambda=0$, then $Q_{(-)}$ has image 
$\textrm{Quad}(E)$. 
\end{lem}

\begin{proof}
To prove (1), note first that 
in the exceptional case ($\s=\id$ in characteristic 2), the form $\omega_q$ is indeed alternating, since we necessarily have $\eps=-1$ in this case, which forces
\[ \omega_q(v,v)=q(v,v)+\eps q(v,v)=0. \]
Now, returning to the general case,  let  $\omega$ be any element of $\textrm{Herm}_{\s,\eps}(E)$,  assumed alternating in the special case.
Arbitrarily pick a basis $\{x_i\}$ of $E$, and define
\[ q(x_i,x_j)=\begin{cases}
\omega(x_i,x_j) &\text{if }i<j,\\
a_i &\text{if }i=j,\\
0 &\text{if }i>j,
\end{cases} \]
extending sesquilinearly to obtain an element $q\in \textrm{Sesq}_\sigma(E)$.
Here $a_i\in\F$ is a scalar chosen so that
\[ a_i+\eps\bar a_i=\omega(x_i,x_i). \]
This is possible in the special case by setting $a_i=0$. In the other cases, applying Proposition~\ref{prop:Lambda} we get 
\[ \Lambda_{min}(\F,\s,-\eps)=\set{a+\eps\bar a}{a\in\F}=\set{b\in\F}{b=\eps\bar b}=\Lambda_{max}(\F,\s,-\eps) \]
and $\omega(x_i,x_i)$ lies in the right hand side, so such an $a_i$ does exists. 
The form $q$ satisfies that $\omega_q=\omega$ in all cases.

For (2), we now assume that $\s=\id$. %and $\Lambda=0$. In particular, $Q:E\to \F$. 
Given a formed space $(E,q)$, we first check that  $Q_q$ is a mod $\Lambda$ quadratic form, i.e. that 
$$B_{Q_q}(v,w):=Q_q(v+w)-Q_q(v)-Q_q(w)$$ is bilinear. But $B_{Q_q}(x,y)=\omega_q\mod \Lambda$, which is bilinear under the assumption that $\s=\id$. 

Conversely, assume that $Q:E\to \F/\Lambda$ is a quadratic form (mod $\Lambda$). We want to show that it is in the image of $Q_{(-)}$.   
Arbitrarily pick a basis $\{x_i\}$ of $E$, and pick $b_{ij},c_i\in \F$ such that $B_Q(x_i,x_j)=b_{ij}\mod \Lambda$ and $Q(x_i)=c_i\mod\Lambda$. 
Then define
\[ q(x_i,x_j)=\begin{cases}
b_{ij} &\text{if }i<j,\\
c_i &\text{if }i=j,\\
0 &\text{if }i>j,
\end{cases} \]
extending bilinearly to obtain an element $q\in \textrm{Sesq}_{\id}(E)$.
For each $i$, we have $Q_q(x_i)=Q(x_i)$ by construction.
Now  $$\omega_q(x_i,x_j)=q(x_i,x_j)+\eps q(x_j,x_i)=q(x_i,x_j)+ q(x_j,x_i)\mod \Lambda$$ as $q(x_j,x_i)=\eps q(x_j,x_i)\mod\Lambda_{min}$. 
Hence $\omega_q(x_i,x_j)=B_Q(x_i,x_j)\mod \Lambda$  if $i\neq j$, 
while
\[ \omega_q(x_i,x_i)=2q(x_i,x_i)=2Q(x_i)\mod\Lambda \]
and
\[ B_Q(x_i,x_i)=Q(2x_i)-2Q(x_i)=2Q(x_i). \]
It follows that $\omega_q=B_Q$, while $Q_q$ and $Q$ coincide on a basis. This implies that $Q_q=Q$ because 
$Q(v+w)=Q(v)+Q(w)+B_Q(v,w)$
and likewise $Q_q(v+w)=Q_q(v)+Q_q(w)+\omega_q(v,w)\mod \Lambda$,
finishing the proof. 
\end{proof}

\begin{lem}\label{lem:epsilonredundant}
If $\s\neq \id$, there exists an $a=a(\eps)\in \F^\times$ such that multiplication by $a$ induces an isomorphism 
$$\Forms(E,\sigma,\eps,\Lambda) \sta{\cong}\rar \Forms(E,\sigma,1,a\Lambda).$$
\end{lem}

Note that $\Gi(E,q)=\Gi(E,aq)$ as subgroups of $\G(E)$.

\begin{proof}
Multiplication by $a\in \F$ takes $X(V,\sigma,\eps,\Lambda)$ to $X(V,\sigma,a\bar a\inv\eps,a\Lambda)$.
So we need to check that there exists an $a\in \F$ such that  $\eps=a\inv\bar a$.  
The existence of such an element  is given by applying Hilbert's Theorem 90 to the field extension $\F$ of $\F^\s$. 
\end{proof}

\begin{proof}[Proof of Theorem~\ref{introthm:qspecialcases}]
We start by proving (1). So assume $\s\neq \id$. Then by Proposition~\ref{prop:Lambda}, we must have $\Lambda=\Lambda_{max}$. Now by Lemma~\ref{prop:inject}(1), the map $\omega_{(-)}$ is injective and by Lemma~\ref{lem:surject} it is also surjective. Statement (1) then follows from Lemma~\ref{lem:epsilonredundant}. 

To prove (2), we now assume that $\s=\id$, $\eps=-1$ and $\Lambda=\F$. In particular, $\Lambda=\Lambda_{max}$ in all cases by Proposition~\ref{prop:Lambda}. 
Applying Lemma~\ref{prop:inject}(1) again, we have that the map $\omega_{(-)}$ is injective and by Lemma~\ref{lem:surject} it is also surjective onto $\textrm{Herm}_{\id,-1}(E)$ unless $\F$ has characteristic $2$ in which case the image is the subgroup of alternating forms $\textrm{Alt}(E)$. Now statement (2) follows from the fact that $\textrm{Herm}_{\id,-1}(E)$ and $\textrm{Alt}(E)$ are actually isomorphic if the field is not of characteristic 2, as $\omega\in \textrm{Herm}_{\id,-1}(E)$ satisfies that $\omega(v,v)=-\omega(v,v)$. 

For (3),  we now assume that $\s=\id$, $\eps=1$ and $\Lambda=0$. By  Proposition~\ref{prop:inject}(2), the map $Q_{(-)}$ is injective, and  by Lemma~\ref{lem:surject} it surjects onto the quadratic forms $\textrm{Quad}(E)$, which proves the result. 
\end{proof}

Theorem~\ref{introthm:qspecialcases} gives a description of almost all the possible formed space in the better-known terms of Hermitian, alternating or quadratic forms. As we saw in Remark~\ref{rem:almostall}, left is 
the case $\sigma=\id$, $p=2$, and $0<\Lambda<\F$.
We show now that at least for $\F$ finite, this case is in fact redundant, producing the same isometry groups as if $\Lambda=0$:

\begin{prop}\label{prop:perfectfield}
Assume $\F$ is a perfect field of characteristic 2.
Let $\sigma=\id$ and $\Lambda<\F$ be proper.  Then
\[ X(E,\id,1,\Lambda)=X(E,\id,1,0). \]
\end{prop}
\begin{proof}
What we must show is that, if $q(v,v)\in\Lambda$ for all $v\in V$, then in fact $q(v,v)=0$ for all $v$. 
For any $c\in\F^\times$, we have 
\[ q(c\inv v,c\inv v)=c^{-2}q(v,v)\in\Lambda. \]
Now, if $q(v,v)$ is nonzero, then every element of $\F$ is of that form for some $c$, because the squaring function on $\F$ is bijective.  That contradicts $\Lambda$ being a proper subset of $\F$.
\end{proof}

\begin{rem}
The proposition is false for imperfect fields of characteristic two.  For instance, for  $\F=\F_2(t)$,   
we have proper inclusions
\[ X(\F,\id,1,0)<X(\F,\id,1,\F_2(t^2))<X(\F,\id,1,\F_2(t)). \]
Indeed, the forms $q(a,b)=ab$ and $q'(a,b)=tab$  demonstrate properness of the two inclusions. 
\end{rem}

\section{Witt's Lemma}\label{sec:Witt}

In this section, we give a proof of Witt's lemma in the context of formed spaces, stated as Theorem~\ref{lem:Witt}, and its relative
version, stated as Theorem~\ref{thm:relativeWitt}. 

\medskip

Given a vector space $E$, recall that $^\s\! E^*$ denotes the vector space of
$\s$-skew-linear maps $E\to \F$, that is, additive maps $f:E\to \F$
such that $f(av)=\bar a f(v)$, with vector space structure defined
pointwise from that of $\F$. 
Recall also  the maps $\omega_q$ and $Q_q$ of Theorem~\ref{thm:qdetermined} associated to a form $q$.

\begin{defn}\label{def:kerneletc} Let $E=(E,q)$ be a formed space. \\
(1) The {\em kernel} of $E$ is defined to be the kernel
$\K(E) = \ker(\flat_q)$ 
of the associated linear map
\[ \flat_q: E\to {^\sigma\! E^*},\ v\mapsto\omega_q(-,v). \]
We say that $(E,q)$ is {\em non-degenerate} if $\K(E)=0$.

\smallskip

\noindent
(2) The \textit{orthogonal complement} 
$U^\perp$ of a subspace $U\leq E$ is  the subspace consisting of all $v\in E$ such that $\omega(v,U)=0$.  That is, $U^\perp$ is the kernel of the composition
\[ E\xrightarrow{\flat_q}{^\sigma\! E^*}\xrightarrow{incl^*}{^\sigma U^*}. \]

\smallskip

\noindent
(3) The \textit{radical} of $E$ is  the set
\[ \R(E)=\set{v\in\K(E)}{Q_q(v)=0}. \]

\smallskip

\noindent
(4) A subspace $U$ of $E$ is called \textit{isotropic} if $q|_U=0$ (or equivalently, $\omega_q|_U=0$ and $Q_q|_U=0$).

\end{defn}

\begin{rem}
(1) The radical $\R(E)$ is in fact a subspace of $E$, because $Q_q$ is additive on $\K(E)$, and satisfies $Q_q(cv)=c^2Q_q(v)$ for $c\in\F,v\in E$.
If the characteristic of $\F$ is not 2, then $\R(E)=\K(E)$, since equation (\eqref{equ:Qomega}) in that case gives that $\omega_q|_{\K(E)}=0$ implies $Q_q|_{\K(E)}=0$. 

\smallskip

\noindent
(2) Note that orthogonality is a symmetric relation: $\omega_q(v,w)=0$ if and only if $\omega_q(w,v)=0$, but note also that the orthogonal complement $U^\perp$ defined above will usually not be a complement in the sense of vector spaces. In fact,  if $U$ is isotropic, we actually have $U\subset U^\perp$! 
\end{rem}

\begin{ex}\label{ex:Cox}
To a Coxeter graph on $n$ vertices, one associates an $(n\x n)$--real
symmetric matrix (see eg., \cite[Sec 2.3]{Hum}), with corresponding symmetric bilinear form on $\mathbb{R}^n$. Hence we get an associated  formed space $(\mathbb{R}^n,q)$ over $\mathbb{R}$ with the parameters $(\s,\eps,\La)=(\id,1,0)$.  
This formed space defines a geometric representation of the corresponding Coxeter group, in the sense that the group acts by isometries on this space  (see eg., \cite[Sec 5.3]{Hum}). 
The form $q$ is non-degenerate if the group is finite, but can be degenerated when the group is infinite; 
see eg., \cite[Sec 2.5]{Hum} for examples of Coxeter graphs whose associated symmetric bilinear forms are degenerate. 
\end{ex}

The following result is a version of Witt's Lemma. Witt originally considered non-degenerate symmetric forms in characteristic not 2 \cite{Witt}. 
His result was generalize by eg., Tits \cite[Prop 2]{Tits} or Artin \cite[Thm 3.9]{Artin}, who considered non-degenerate alternating and quadratic forms, 
Dieudonn\'e \cite[p21,22 and 35]{Dieudonne}, 
who considers non-degenerate hermitian and quadratic forms, 
and Magurn-Vaserstein-van der Kallen \cite[Cor 8.3]{MVV} or  Petrov \cite[Thm 1]{Petrov} who both work in the more general context of rings and have a different set of assumptions, including an assumption on the {\em index} of the form, which counts the number of hyperbolic summands it contains.   The paper \cite{Huang} of Huang considers flags of subspaces in the case of non-degenerate symmetric forms in characteristic not 2.

\begin{thm}[Witt's Lemma]\label{lem:Witt}
Let $(E,q)$ be a formed space.
Suppose that $f:U\to W$ is a bijective isometry between two subspaces of $E$ such that $f\big(U\cap\K(E)\big)= W\cap\K(E)$. 
Then $f$ can be extended to a bijective isometry of $E$. 
In particular, the group $\Gi(E,q)$ of bijective isometries of $(E,q)$ acts transitively on the set of isotropic subspaces $U$ with given values of $\dim U$ and $\dim U\cap\R(E)$.
\end{thm}

Regarding the last statement, observe that when $U$ is isotropic, $U\cap\R(E)=U\cap\K(E)$. 
Note also that for general $U,W$ and $f$ as in the statement, we have that $f(U\cap\R(E))=W\cap\R(E)$ as $f$ is an isometry that takes $U\cap\K(E)$ to $W\cap\K(E)$.

The result implies that the group $\Gi(E,q)$ acts transitively on the elements of any fixed rank of the building of isotropic subspaces
\[ \PPi(E) = \set{W<E}{W\textrm{ isotropic},\ \R(E)<W}. \] 
This transitivity is used in \cite{Classical} to study the homology of the group $\Gi(E,q)$.

\medskip

The proof of the theorem will use the following  basic result about vector spaces: 

\begin{lemma}\label{lem:simultaneous complement}
Let $A,B\leq V$ be any two subspaces of the same dimension.  Then there exists a subspace $L$ which is simultaneously complementary to both $A$ and $B$.
\end{lemma}

Notice this would be false with three subspaces instead of two in e.g.~$V=\F_2^2$.
\begin{proof}
By modding out $A\cap B$ from $V$, we may assume without loss of generality that
$A\cap B=0$. 
By restricting to $A+B<V$, we may assume that $A+B=V$.
Pick any linear isomorphism
$f:A\to B$, 
and define $g:A\to V$ by setting $g(u)=u+f(u)$.
We claim that  $L=\im(g)\leq V$
is the desired complement.
Now, $A\cap L=0$ because $f$ is injective and $A,B$ are disjoint.
Similarly $B\cap L=0$.  Also $B+L=V$ since it contains $A$, and
$A+L=V$ since it contains $\im(f)=B$, which finishes the proof.
\end{proof}

In the proof of Witt's Lemma, we will repeatedly use the following construction: 

\begin{lem}\label{lem:f+}
Suppose $A,B,C\le E$ are subspaces of a formed space $(E,q)$ such that $C\le A^\perp\cap B^\perp$ and $A\cap C=0=B\cap C$. Then for any isometry $f:A\to B$, the map $f\op id:A\op C\to B\op C$ is also an isometry.  If $A=\langle a\rangle$, $B=\langle b\rangle$ and $f(a)=b$, the same holds under the weaker assumption that $C\le (a-b)^\perp$. 
\end{lem}

\begin{proof}
In both cases, it is easy to see that $f\op id$ preserves $\omega_q$. By Theorem~\ref{thm:qdetermined}, we thus just have to check that $f$ preserves $Q_q$. In the first case, we have 
$Q_q(f(a)+c)=Q_q(f(a))+Q_q(c)+\omega_q(f(a),c)=Q_q(a)+Q_q(c)$. Similarly,  $Q_q(a+c)=Q_q(a)+Q_q(c)$. In the second case, we have instead 
$Q_q(f(a)+c)=Q_q(b)+Q_q(c)+\omega_q(b,c) =Q_q(a)+Q_q(c)+\omega_q(a,c)=Q_q(a)+Q_q(c)$. This proves the result in both cases. 
\end{proof}

\begin{proof}[Proof of Theorem~\ref{lem:Witt}] 
We will prove the theorem by induction on the dimension of $U$, starting with the case 
\[ U\le K:=\K(E). \]
By our hypothesis, it follows that $W\le K$
as well.
By Lemma~\ref{lem:simultaneous complement}, we can find a subspace $L\leq E$ such that
\[ U\oplus L=E=W\oplus L. \]
We extend $f$ to a bijection of $E$ by setting $f$ to be identity on $L$.  As $U^\perp=W^\perp=E$, Lemma~\ref{lem:f+} shows that this is an isometry. 

Now, we may assume that
\[ U\not\leq K. \]
Then we can pick a codimension-one subspace $H<U$, such that $$U\cap K\leq H.$$
By induction hypothesis we may find an isometry of $E$ which coincides with $f$ on $H$; by composing $f$ with the inverse of such an isometry, we reduce to the case where
\[ f|_H=id_H. \]
Write
\[ U=H\oplus\langle a\rangle  \ \ \ \textrm{and}\ \  \ W=H\oplus\langle b\rangle \ \ \ \textrm{with}\ \ \ b=f(a), \]
noting that $a,b\notin K$ by our choice of $H$. 
The subspace $(b-a)^\perp\leq E$  has codimension at most 1 since it is the kernel of a linear map with 1-dimensional codomain. 

\medskip

\noindent
\textbf{Case 1:  $(b-a)^\perp=E$.}  Then, by applying Lemma~\ref{lem:simultaneous complement} again to the nonzero vectors $a$ and $b$ in the vector space $E/H$, we obtain a hyperplane $M$ with $H\leq M<E$, which is complementary to both $\langle a\rangle$ and $\langle b\rangle$.   By Lemma~\ref{lem:f+},  
$f$ extends to a bijective isometry on all of $E$ by declaring it to be the identity on $M$ as $M\le (b-a)^\perp$.

\medskip

\noindent
\textbf{Case 2:  $(b-a)^\perp<E$}
is a hyperplane.
Note that 
\[ H\leq (b-a)^\perp, \]
since $\omega_q(h,a)=\omega_q(h,b)$ for all $h\in H$.
Also, we have that $a\in(b-a)^\perp$ if and only if $b\in(b-a)^\perp$ because
$\omega_q(b,b)=\omega_q(a,a)$, so
\begin{align*}
\omega_q(b,b-a) &= \omega_q(b,b)-\omega_q(b,a) = \omega_q(a,a)-\omega_q(b,a) = \omega_q(a-b,a).
\end{align*}

\medskip

\noindent
\textbf{Case 2.1: $a,b \notin (b-a)^\perp$.}
This means that
\[ (b-a)^\perp\oplus\langle a\rangle = E = (b-a)^\perp\oplus\langle b\rangle. \]
As in case one, we can extend $f$ by defining it to be the identity on $(b-a)^\perp$ using again Lemma~\ref{lem:f+} as this space is by definition orthogonal to $b-a$.

\medskip

\noindent
\textbf{Case 2.2: $a,b \in (b-a)^\perp$.}
We first check that this implies that
\[ Q_q(b-a)=0. \]
Indeed, if the characteristic of $\F$ is two, this is true because
\begin{align*}
Q_q(b-a) \, &=\,  Q_q(b)+Q_q(a)-\omega_q(b,a) \, \\
&=\,  2Q_q(b)-\omega_q(a,a) 
\, =\,  q(a,a)-\eps\overline{q(a,a)} \,  = 0 \mod\Lambda.
\end{align*}
If the characteristic of $\F$ is not 2, instead we observe that
\[ \omega_q(b-a,b-a)=0, \]
and then recall that, since $\Lambda_{\min}=\Lambda=\Lambda_{\max}$ in this case by Proposition~\ref{prop:Lambda}, the vanishing of 
$\omega_q(b-a,b-a)$ implies that of $Q_q(b-a)$, as in the proof of Lemma~\ref{prop:inject}(1).

We have $H,K\le (b-a)^\perp$ and by the additional assumption also $U=H\op \langle a\rangle, W=H\op \langle b\rangle$ are also subspaces of $(b-a)^\perp$. 
Recall also that $K\cap U=K\cap H=K\cap W$ by our choice of $H$. Let $L\le (b-a)^\perp$ by a common complement of $U$ and $W$ inside $(b-a)^\perp$ that contains $K$. (Such a complement can be constructed by first picking a complement $N$ of $H\cap K$ inside $K$ and then using  Lemma~\ref{lem:simultaneous complement} on $(b-a)^\perp/N$.)  
Setting $M=H\op L\le (b-a)^\perp$, we have 
$$(b-a)^\perp=M\op a=M\op b,$$
and by construction $K,H\le M$.

As $M$ is orthogonal to $b-a$, we can again extend $f$ to an isometry of $(b-a)^\perp$ by setting it to be the identity map on $M$. We need to extend $f$ to the whole of $E$.

Recall that $a,b\notin K$, so both
\[ a^\perp,b^\perp<E \]
are hyperplanes.
We claim that neither of them contains $M^\perp$.
Assume to the contrary that
\[ M^\perp\subset a^\perp. \]
That implies that
$ a\in a^{\perp\perp}\leq M^{\perp\perp}=M+K=M$, 
a contradiction since $a$ is not in $M$.  Similar logic applies to $b$.  So $M^\perp$ is contained in neither $a^\perp$ nor $b^\perp$.  So,
\[ M^\perp\cap a^\perp,\ M^\perp\cap b^\perp<M^\perp \]
have codimension one.  Using Lemma~\ref{lem:simultaneous complement} again, there is some vector
\[ v\in M^\perp \]
which is in neither of them.  That is, $v$ not orthogonal to either $a$ or $b$.
We break into two cases depending on whether or not $v$ is orthogonal to $b-a$.

\medskip

\noindent
\textbf{Case 2.2.1:  $v\notin(b-a)^\perp$.}
We intend to extend $f$ to all of $E$ by setting
\[ f(v) = c(b-a)+dv \]
for some scalars $c,d\in\F$ with $d\neq0$.  This will be bijective, so we are done if we can choose scalars yielding an isometry.
Note that
$f(v)-v\in M^\perp$
so
\[ \omega_q(m,v)=\omega_q(m,f(v)) \]
for all $m\in M$.
Now,
\[ \omega_q(b,f(v)) = d\omega_q(b,v). \]
Since $v$ is not orthogonal to $a$ or $b$,
\[ \omega_q(b,v),\omega_q(a,v)\neq0, \] so we can set
$$ d=\frac{\omega_q(a,v)}{\omega_q(b,v)} \neq0 \ \ \ 
\textrm{to get}
\ \ \ \omega_q(b,f(v))=\omega_q(a,v) \ \ \ \textrm{as needed.}$$
Lastly,
\[ Q_q(f(v)) = c\bar cQ_q(b-a) + d\bar dQ_q(v) + c\bar d\omega_q(v,b-a). \]
The first term is zero.  In the last term, both $d$ and $\omega_q(v,b-a)$ are nonzero (since we have assumed that $v$ is not orthogonal to $b-a$.  So we can pick
\[ c=\frac{(1-d\bar d)Q_q(v)}{\bar d\omega_q(v,b-a)}
\ \ \ \textrm{to get} \ \ \  Q_q(f(v))=Q_q(v). \]
This shows that $f$ is an isometry of $E$.

\medskip

\noindent
\textbf{Case 2.2.2:  $v\in (b-a)^\perp$.}
Choose a complement
\[ z\notin(b-a)^\perp. \]
This time we intend to extend $f$ to all of $V$ by setting
\[ f(z) = z+c(b-a)+dv \]
for some scalars $c,d\in\F$.  This will be bijective, so we are done if we can choose scalars yielding an isometry.
Again,
$f(z)-z\in M^\perp$ 
so
\[ \omega_q(m,z)=\omega_q(m,f(z)) \]
for all $m\in M$.
Now,
\[ \omega_q(b,f(z)) = \omega_q(b,z)+d\omega_q(b,v). \]
Since $v$ is not orthogonal to $b$,
we can set
\[ d=\frac{\omega_q(a,z)-\omega_q(b,z)}{\omega_q(b,v)} 
\ \ \ \textrm{to get}\ \ \ 
 \omega_q(b,f(z))=\omega_q(a,z) \ \ \ \textrm{as needed.}\]
Lastly,
$Q_q(f(z))$ \\
$= Q_q(z) + c\bar cQ_q(b-a) + d\bar dQ_q(v)  + c\omega_q(z,b-a) + d\omega_q(z,v) + \bar cd\omega_q(b-a,v)$ \\
$= Q_q(z) + d\bar dQ_q(v) + c\omega_q(z,b-a) + d\omega_q(z,v),$\\
since we have assumed that $v$ is orthogonal to $b-a$.
But $z$ is not orthogonal to $b-a$, so we can set
\[ c = -\frac{d\bar dQ_q(v)+d\omega_q(z,v)}{\omega_q(z,b-a)} \ \ \
\textrm{to get}\ \ \ 
 Q_q(f(z))=Q_q(z). \]
Using Theorem~\ref{thm:qdetermined}, this shows that $f$ is an isometry of $E$, which finishes the proof.
\end{proof}

The following result is a relative variant of Witt's Lemma, which tells us when
we can extend an isometry between two subspaces to an isometry of the
ambient space which is the identity on a fixed subspace $A$:

\begin{thm}[Relative Witt Lemma]\label{thm:relativeWitt}
Let $(E,q)$ be a formed space and $A\le E$ a subspace.
Suppose that $f:U\to W$ is a bijective isometry between two subspaces of $E$ such that 
\begin{enumerate}
\item $U\cap A=W\cap A$ and $f|_{U\cap A}=\id$;  
\item $f=\id \mod A^\perp$, i.e.~$f(u)-u\in A^\perp$ for any $u\in U$; 
\item $f\big(U\cap\K(E)\big)= W\cap\K(E)$. 
\end{enumerate}
Then $f$ can be extended to a bijective isometry of $E$ that restricts to the identity on $A$.
\end{thm}

The following result shows that the assumption that $f$ is the identity modulo $A^\perp$ is actually necessary.

\begin{lem}\label{lem:Vperpfix} Let $(E,q)$ be a formed space and $A\leq E$ any subspace.
Any isometry of $E$ fixing $A$ pointwise preserves $A^\perp$ and is the identity modulo $A^\perp$.
\end{lem}

\begin{proof} 
Let $f:E\to E$ be an isometry fixing $A$ pointwise.
If $v\in E$ and $a\in A$, then
\[ \omega_q(f(v),a)=\omega_q(f(v),f(a))=\omega_q(v,a), \]
so $f(v)-v\in A^\perp$, showing both that $f$ preserves $A^\perp$ and that it is the identity modulo $A^\perp$.
\end{proof}

\begin{proof}[Proof of Theorem~\ref{thm:relativeWitt}]
Let $A_0$ be a complement of $U\cap A=W\cap A$ inside $A$. We claim that the map $\bar f=f\op \id: U\op A_0\to W\op A_0$ is an isometry. Indeed, $\bar f$ restricts to an isometry on both $U$ and $A_0$. Let $u\in U$ and $a\in A_0$.  By assumption, we have $f(u)=u+v$ with $v\in A^\perp$. Hence $\omega_q(f(u),f(a))=\omega_q(u+v,a)=\omega_q(u,a)$, so $\bar f$ preserves $\omega_q$. Also, $Q_q(\bar f(u+a))=Q_q((u+v)+a)=Q_q(f(v))+Q_q(a)+\omega_q(u+v,a)=Q_q(v)+Q_q(a)+\omega_q(u,a)=Q_q(u+a)$ where we used that $f$ is an isometry. Hence, by Theorem~\ref{thm:qdetermined}, $\bar f$ is an isometry. As $\bar f((U\op A_0)\cap \K(E))=f(U\cap \K(E))\op (A_0\cap \K(E))=(W\op A_0)\cap \K(E)$, we can apply Theorem~\ref{lem:Witt} to extend $\bar f$ to a bijective isometry $\hat f$ of $E$ that takes $U\op A_0$ to $W\op  A_0$. By construction, $\hat f$ extends $f$ and restricts to the identity on $A$, which proves the result. 
\end{proof}

The following corollary of the relative Witt's Lemma was our motivation for writing the result: 

\begin{cor}\label{cor:Wittisotropic}
Let $(E,q)$ be a formed space, and $U,W_1,W_2$ isotropic subspaces containing $\R(E)$. 
Suppose $\dim W_1=\dim W_2$ and $W_1+U^\perp=W_2+U^\perp=E$. Then there is a bijective isometry
%\[ U_1+V_0 \iso U_2+V_0 \]
of $E$ sending $W_1$ to $W_2$, and restricting to the identity on $U$.
Furthermore, an isomorphism $f: W_1\to W_2$ restricting to the identity on $\R(E)$ can be extended to an isometry of $E$ fixing $U$ pointwise if and only if $f$ is the identity map modulo $U^\perp$.
\end{cor}

For an isotropic subspace $U$ containing the radical $\R(E)$, define the relative building to be
\[ \PPi(E,U) = \set{W<E}{W\textrm{ isotropic},\ \R(E)<W,\ W+U^\perp=E}. \]
The corollary implies in particular that the group 
\[ \Ai(E,U) := \set{g\in\Gi(E,q)}{g|_{U}=\textrm{id}_{U}} \]
of the bijective isometries of $(E,q)$ fixing $U$ pointwise acts transitively on the elements of any given fixed rank. This is used in \cite{Classical} to study the homology of the groups $\Ai(E,U)$.

The proof of the corollary  will use the following: 
\begin{lem}\label{lem:intV0} Let $(E,q)$ be a formed space and let
$U,W\le E$ be subspaces. If $U$ is isotropic and $W+U^\perp=E$, then $U\cap W^\perp\leq\R(E)$. In particular, if $W$ is also isotropic, then $U\cap W\leq\R(E)$. 
\end{lem}

\begin{proof}[Proof of Lemma~\ref{lem:intV0}]
Suppose that $U\cap W^\perp=A$. As $A\subset W^\perp$, we also have $W\subset A^\perp$. And $A\subset U$ implies that $U^\perp\subset A^\perp$. So $E=W+U^\perp\subset A^\perp$. 
This means that 
$A\subset\K(E)$. But $A\subset U$ is 
isotropic, so $Q_q|_A=0$ and we conclude $A\leq\R(E)$ as claimed.
The last claim follows from the fact that $W\subset W^\perp$ when $W$ is isotropic. 
\end{proof}

\begin{proof}[Proof of Corollary~\ref{cor:Wittisotropic}]
We want to apply Theorem~\ref{thm:relativeWitt}. For this, we need to construct a bijective isometry $f:W_1\to W_2$ satisfying the assumptions of the theorem. As $W_1$ and $W_2$ are isotropic, any isomorphism will give a bijective isometry. 
Note now that, by Lemma~\ref{lem:intV0}, $W_i\cap U\le \R(E)$, and hence $W_i\cap U=\R(E)$ as $\R(E)$ is assumed to be contained in $U,W_1$ and $W_2$. This gives condition (1) in the theorem. As $W_1+U^\perp=E=W_2+U^\perp$, we can pick a isomorphism $f$ which is the identity modulo $U^\perp$,  and hence satisfies condition (2). As  $W_1 \supset \R(E) \subset W_2$, we can choose such an $f$ which is the identity on $\R(E)$. This shows that condition (3) is also satisfied as $W_i\cap \K(E)=W_i\cap \R(E)=\R(E)$ as each $W_i$ is isotropic. 
Hence there exists a bijective isometry of $E$ taking $W_1$ to $W_2$ and restricting to the identity on $U$. 

For the last statement, Lemma~\ref{lem:Vperpfix} shows that any isometry  of $E$ that fixes $U$ necessarily is the identity modulo $U^\perp$, which gives one direction. On the other hand, if $f$ is the identity on $\R(E)$ and the identity modulo $U^\perp$, we see that it satisfies the conditions of Theorem~\ref{thm:relativeWitt} as it is automatically an isometry, satisfies conditions (2) and (3) by assumption, and condition (1) is satisfied by the above argument, which does not depend on $f$.   
\end{proof}

\appendix

\section{The possible values of $\Lambda$}\label{sec:Lambda}

Let $(\F,\s)$ be as above a field with chosen involution, where we again write  $\bar c=\sigma(c)$.
If $\sigma$ is not the identity, then \[ \F\supset\F^\sigma \]
is a field extension of degree 2, where $\F^\s$ denotes the fixed points of the involution.
We fix as before a scalar $\eps\in\F$ satisfying $\eps\bar\eps=1$.

Recall that 
\[ \Lambda_{min} = \set{a-\eps\bar a}{a\in\F} \ \ \ \ \textrm{and}\ \ \ \
 \Lambda_{max} = \set{a\in\F}{a+\eps\bar a=0}. \]
Note first that $\Lambda_{min}\le\Lambda_{max}$, because
$(a-\eps\bar a)+\eps\overline{(a-\eps\bar a)} = 
% a-\eps\bar a+\eps\bar a-\bar\eps\eps a
0$. 

The following result gives the values of $\Lambda_{min}$ and $\Lambda_{max}$ in the case of fields, depending on whether the involution is trivial or not, the characteristic of the field is even or not, and whether $\eps$ is equal to $1,-1$ or something else. 

\begin{prop}\label{prop:Lambda}
Let $\F$ be a field of characteristic $p$, $\s$ an involution on $\F$ and $\eps\in \F$ such that $\eps\bar\eps=1$. Then the values of $\Lambda_{min}$ and $\Lambda_{max}$ are given by the following table: 
\[
  \begin{array}{ c|c|c||c|c  }
	\sigma=\id? & p=2? & \eps & \Lambda_{min} & \Lambda_{max} \\
    \hline
    yes&no&1 & 0&0 \\ \hline
    yes&no&-1 & \F&\F \\ \hline
    yes&yes&1 & 0&\F \\ \hline
    no& &-1 & \F^\sigma & \F^\sigma \\ \hline
	no& &\neq -1 & (1+\eps)\F^\sigma & (1+\eps)\F^\sigma.
  \end{array}
\]
\end{prop}

Note in particular that in most cases $\Lambda_{min}=\Lambda_{max}$ in which case there is only one possible choice for $\Lambda_{min}\le \Lambda\le \Lambda_{max}$. The only situation where $\Lambda_{min}\neq\Lambda_{max}$ is when $\s$ is the identity and the characteristic is equal to 2, in which case $\Lambda$ can be any additive subgroup of $\F$. 

\begin{rem}
In the Hermitian case, we take $\s\neq\id$ and $\eps=1$. If $\operatorname{char}(\F)=2$, then $\eps=1=-1$ and we see that $\La=\F^\s$, while if $\operatorname{char}(\F)\neq 2$, then $\eps=1\neq -1$ and we see that $\La=2\F^\s$, which also identifies with $\F^\s$ as $2$ is invertible in that case. 
\end{rem}

\begin{proof}
Suppose first that $\sigma=\id$.  This implies that
\[ \eps=\pm1. \]
(In characteristic 2, this gives only one possibility.) If $\eps=1$, then $\Lambda_{min}=0$.  Otherwise, $(1-\eps)$ is invertible and $\Lambda_{min}=\F$. This gives $\Lambda_{min}$ in the first three cases in the table. 
If $\eps=-1$, we have that $\Lambda_{max}=\F$.  Otherwise, $(1+\eps)$ is invertible and $\Lambda_{max}=0$.  This completes the proof of the first three cases in the table.

Now consider the case when $\sigma$ is nontrivial.
The set $\Lambda_{min}$ is the image of the map
\[ \F\to\F,\quad a\mapsto a-\eps\bar a \]
This map is $\F^\sigma$-linear, so its image is an $\F^\sigma$-subspace of the two-dimensional vector space $\F$ (over $\F^\sigma$).
We claim that this map has rank 1.

First assume that $\eps\neq -1$.  Then its kernel contains the element
$(1+\eps)$, so it is not injective.
However, it cannot be identically zero, because (by applying the map to $1\in\F^\sigma$) that would imply $\eps=1$ and $\sigma=id$, contradiction our assumption.
Next assume that $\eps=-1$.  Then our map $a\mapsto a+\bar a$ is the trace map of the field extension.  This map is well-known to be nonzero from character theory. 
It cannot be surjective since its image lies in $\F^\sigma$.
In either case, we conclude that $\Lambda_{min}$ is a one-dimensional subspace of $\F$ over $\F^\sigma$. 

We will compute $\Lambda_{max}$ to identify $\Lambda_{min}$. By definition, $\Lambda_{max}$ is the kernel of the map
\[ \F\to\F^\sigma,\quad a\mapsto a+\eps\bar a. \]
Since this is the same map as before but with $\eps$ replaced by $-\eps$, we have already showed that it has rank 1 as an $\F^\sigma$-linear endomorphism of $\F$.
So its kernel is a one-dimensional subspace.  Because it must contain $\Lambda_{min}$, we have
\begin{align*}
\Lambda_{max} &= \Lambda_{min} = \begin{cases} \F^\sigma &\text{if }\eps=-1,\\
(1+\eps)\F^\sigma&\text{if }\eps\neq -1. \end{cases}
\end{align*}
as subsets of $\F$.
\end{proof}

\subsection*{Acknowledgements}
The authors were both supported by the Danish National Research Foundation through
the Centre for Symmetry and Deformation (DNRF92). The second author was also supported by the  European Research
Council (ERC) under the European Union's Horizon 2020 research and
innovation programme (grant agreement No.~772960), the Heilbronn Institute
for Mathematical Research, and 
would like to thank the Isaac Newton Institute for Mathematical Sciences, Cambridge, for support and hospitality during the programme Homotopy Harnessing Higher Structures, where this paper was finalized (EPSRC grant numbers EP/K032208/1 and EP/R014604/1).

\bibliographystyle{mscplain.bst}
\bibliography{biblio}

\end{document}